\documentclass[12pt]{amsart}
\usepackage{amssymb}
\usepackage{hyperref}
\usepackage{xcolor}

\newtheorem{theorem}{Theorem}[section]
\newtheorem{corollary}[theorem]{Corollary}
\newtheorem{proposition}[theorem]{Proposition}
\newtheorem{lemma}[theorem]{Lemma}

\theoremstyle{definition}

\newtheorem{example}[theorem]{Example}

\newtheorem{remark}[theorem]{Remark}

\def\Aut{{\rm Aut}}
\def\Out{{\rm Out}}
\def\soc{{\rm soc}}
\def\AGL{{\rm AGL}}
\def\SL{{\rm SL}}
\def\GL{{\rm GL}}
\def\PSL{{\rm PSL}}
\def\PSigmaL{{\rm P\Sigma L}}
\def\PGammaL{{\rm P\Gamma L}}
\def\PGL{{\rm PGL}}
\def\Sym{{\rm Sym}}
\def\Alt{{\rm Alt}}

\def\ker{{\rm ker}\,}

\newcommand{\C}{\mathrm{C}}
\newcommand{\D}{\mathrm{D}}

\newcommand{\N}{\mathrm{N}}

\newcommand{\M}{\mathrm{M}}

\newcommand{\NN}{\mathbb N}
\newcommand{\ZZ}{\mathbb Z}

\begin{document}
\title{Skew product groups for monolithic groups}
\date{\today}
\author{Martin Bachrat{\' y}}
\address{Mathematics Department, University of Auckland, PB 92019, Auckland, New Zealand}
\email{mbac631@aucklanduni.ac.nz}
\author{Marston Conder}
\address{Mathematics Department, University of Auckland, PB 92019, Auckland, New Zealand}
\email{m.conder@auckland.ac.nz}
\author{Gabriel Verret}
\address{Mathematics Department, University of Auckland, PB 92019, Auckland, New Zealand}
\email{g.verret@auckland.ac.nz}

\begin{abstract}
Skew morphisms, which generalise automorphisms for groups, provide a fundamental tool for the study of regular Cayley maps and, more generally, for finite groups with a complementary  factorisation $G=BY$, where $Y$ is cyclic and core-free in $G$.  In this paper, we classify all examples in which $B$ is monolithic (meaning that it has a unique  minimal normal subgroup, and that subgroup is not abelian) and core-free in $G$. As a consequence, we obtain a classification of all proper skew morphisms of finite non-abelian simple groups.
\end{abstract}

\maketitle

\section{Introduction}

Let $G$ be a finite group expressible as a product $G = BC$ of subgroups $B$ and $C$ such that $B \cap C = \{1\}$, sometimes called a \emph{complementary factorisation} of $G$. If $C$ is cyclic, generated by $y$, say, then for every $b \in B$ there exists a unique $b' \in B$ and a unique $j \in \{1,2,\ldots,|C|\!-\!1\}$ such that $yb = b'y^j$.  This induces a bijection $\varphi\!: B \to B$ and a function $\pi\!: B \to \NN$, defined by $\varphi(b) = b'$ and $\pi(b) = j$, having the properties that $\varphi(1_B) = 1_B$ and $\varphi(ab) = \varphi(a)\varphi^{\pi(a)}(b)$ for all $a,b \in B$. Any bijection $\varphi\!: B \to B$ with these two properties is called a \emph{skew morphism} of $B$, and $\pi$ is its associated \emph{power function}.  This is clearly a generalisation of the concept of a group automorphism, which occurs in the special case where $\pi(b) =1$ for all $b$.

Conversely, let $\varphi$ be a skew morphism of a group $B$. Then $\varphi\in\Sym(B)$ and, identifying $B$ as a subgroup of $\Sym(B)$ through its left regular representation, it can be easily verified that  $G=B\langle\varphi\rangle$ is a group, called the \emph{skew product group} associated with $B$ and $\varphi$ (see \cite{KovacsNedela1}).  Moreover, one can show that $G=B\langle\varphi\rangle$ is a complementary factorisation and that $\langle \varphi \rangle$ is core-free in $G$ (see ~\cite[Lemma 4.1]{ConderJajcayTucker}). To find all skew morphisms of a group $B$, it thus suffices to find all complementary factorisations $G = B\langle y\rangle$  with $\langle y\rangle$ core-free in $G$.

Skew morphisms were introduced and found to be important in the study of \emph{regular Cayley maps}, which are embeddings of Cayley graphs on surfaces with an arc-transitive automorphism group that contains a vertex-regular subgroup; see \cite{JajcaySiran}.  In particular, they arose as an essentially algebraic concept with applications in topological graph theory.
The theory of skew morphisms has been developed over the last 16 years, following the publication of \cite{JajcaySiran}.
For example, it is now known that the order of a skew morphism $\varphi$ of a non-trivial group $B$ is always less than $|B|$; see \cite[Theorem 4.2]{ConderJajcayTucker}.

It is easy to find all skew morphism of groups of very small order (see \cite{ConderList} for example). On the other hand, it has proved challenging to determine the skew morphisms of infinite classes of finite groups.  In particular, this problem remains open for finite cyclic groups, despite positive progress recently (as in \cite{KovacsNedela1,KovacsNedela2} for example) and the determination of the corresponding regular Cayley maps for cyclic groups (see \cite{ConderTucker}).  Indeed the skew morphisms of cyclic groups of order $64$ are still not known. We understand, however, that a full classification of skew morphisms for dihedral groups (by Kov{\' a}cs and Kwon) is imminent, together with a full classification of regular Cayley maps for dihedral groups (by Kan and Kwon), building on earlier work (in \cite{KovacsKwon,ZhangDu} for example).

We refer to skew morphisms that are not automorphisms as \emph{proper} skew morphisms. In the cases that have been worked out, there are usually many proper skew morphisms. 
In contrast, we show that non-abelian simple groups rarely have any proper skew morphisms. In fact we take this even further, to monolithic groups. A group $B$ is \emph{monolithic} if it has a unique minimal normal subgroup $A$, and this subgroup $A$ is non-abelian. In that case, $A$ is called the \emph{monolith} of $B$.  Also a subgroup $H$ of a group $G$ is said to be \emph{core-free} in $G$ if it contains no non-trivial normal subgroup of $G$.

\smallskip
Our main theorems are the following:

\begin{theorem}
\label{thm:main}
Let $G$ be a finite group with a complementary factorisation $G=BY$, where $B$ is monolithic,  and $Y$ is cyclic and core-free in $G$. If $B$ is core-free in $G$, then one of the following occurs$\,:$
\begin{enumerate}
\item $G \cong \Sym(n)$, $B \cong \Sym(n-1)$ and $Y\cong \C_n$ for some $n\geq 6$,
\item $G \cong \Alt(n)$, $B \cong \Alt(n-1)$ and $Y\cong \C_n$ for some odd $n\geq 7$,
\item $G \cong \PSL(2,11)$, $B\cong \Alt(5)$ and $Y\cong \C_{11}$,
\item $G \cong \M_{11}$,  $B \cong \M_{10}$ and $Y\cong \C_{11}$,
\item $G \cong \M_{23}$, $B \cong \M_{22}$ and $Y\cong \C_{23}$.
\end{enumerate}
On the other hand, if $B$ is not core-free in $G$, then the monolith $A$ of $B$ is normal in $G$, and either the centraliser $Z = \C_Y(A)$ is trivial, or its order $|Z|$ is less than $|B/A|$.
\end{theorem}

\begin{corollary}
\label{cor:main}
Let $G$ be a finite group with a complementary factorisation $G=BY$, where $B$ is non-abelian and simple, and $Y$ is cyclic and core-free in $G$. Then either $B$ is normal in $G$, or one of the cases $(2)$, $(3)$ or $(5)$ from Theorem~$\ref{thm:main}$ occurs.  Hence in particular, every skew morphism of a non-abelian simple group $B$ is an automorphism, unless $B$ is $\Alt(5)$, $\M_{22}$ or $\Alt(k)$ for some even $k \ge 6$.
\end{corollary}

Here we note that the upper bound on the order of the centraliser $Z = \C_Y(A)$ given in Theorem~\ref{thm:main} is important, for a reason we can explain as follows.  First, since $G=BY$ and $B \cap Y = \{1\}$ we have
\begin{equation*}
|G|= |B||Y| = |B||Y/Z||Z|.
\end{equation*}
Now for a given non-abelian group $A$, there are only finitely many possibilities for a monolithic group $B$ with monolith $A$, because $B$ must be isomorphic to a subgroup of $\Aut(A)$, and in particular, $|B| \le |\Aut(A)|$. Moreover, if $B$ is not core-free in $G$, then $A$ is normal in $G$ and so $Y/Z = Y/\C_Y(A)$ acts faithfully on $A$.  It follows that $|Y/Z|$ is bounded above by the maximum order of an element in $\Aut(A)$. Thus any upper bound on $|Z|$ gives also an upper bound on $|G|$ in terms of $\Aut(A)$, and hence in terms of $A$. See for example Section~\ref{sec:newnew}, where we apply this to determine all skew morphisms of many monolithic groups.

Also we note that using Corollary~\ref{cor:main}, we can find all proper skew morphisms of all simple groups  (see Section \ref{sec:skewsimple}). On the other hand, we cannot say much about the regular Cayley maps for these groups. This contrasts with the situation for cyclic groups, for which the regular Cayley maps were completely classified in \cite{ConderTucker}, but the skew morphisms are still not all known.

\smallskip
The proof of Theorem \ref{thm:main} is presented in Section \ref{sec:proof}, after some more background in Section~\ref{sec:background}.  Then in Section~\ref{sec:remarks} we give a variety of interesting examples regarding Theorem \ref{thm:main}, including a family of examples that show that our upper bound on the order of $Z = \C_Y(A)$ is sharp. Finally, in Section~\ref{sec:skew} we explain how to determine (and count) the skew morphisms of non-abelian simple groups and more general monolithic groups, using Theorem \ref{thm:main}.

\section{Further background}
\label{sec:background}

In this section we give some more background from group theory, and from the theory of skew morphisms.

All groups in this paper are assumed to be finite. We let $\C_n$ denote the cyclic group of order $n$, and $\Alt(n)$ and $\Sym(n)$ the alternating and symmetric groups of degree $n$, and $\M_n$ the Mathieu group of degree $n$.

The \emph{core} of a subgroup $H$ in a group $G$ is the largest normal subgroup of $G$ contained in $H$, and so $H$ is core-free in $G$ if and only if the core of $H$ in $G$ is trivial.

A group $B$ is \emph{almost simple} if there is a non-abelian simple group $A$ such that $B$ can be embedded between $A$ and its automorphism group, that is, $A \le B \le \Aut(A)$.

Every minimal normal subgroup $N$ of a finite group $G$ is a direct product of isomorphic simple groups (this is, $N \cong T^k = T\times\! \stackrel{k}{\dots} \!\times T$ for some simple group $T$ and some positive integer~$k$); see~\cite[Item 3.3.15]{Robinson}. Note that $T$ is cyclic of some prime order $p$ whenever $N$ is soluble, and in that case $N$ is an elementary abelian $p$-group, while otherwise $N$ is non-abelian. The \emph{socle} of a finite group $G$ is the subgroup generated by all of the minimal normal subgroups of $G$, and denoted by $\soc(G)$.  In particular, if $G$ is monolithic with monolith~$A$, then $\soc(G) = A$. The \emph{soluble radical} of $G$ is the largest soluble normal subgroup of $G$.

\medskip
Next, recall that a skew morphism $\varphi$ of a group $B$ is a permutation $\varphi: B \rightarrow B$ fixing the identity element of $B$ and having the property that $\, \varphi(ab) = \varphi(a) \varphi^{\pi(a)}(b)\,$ for all $a,b\in B$, where $\pi: B \rightarrow \NN$ is its associated power function.

The \emph{order} $|\varphi|$ of $\varphi$ is the smallest positive integer $k$ for which $\varphi^k$ is the identity permutation on $B$.   We can thus view $\pi$ as a function from $B$ to $\ZZ_{|\varphi|}$.

The set of all the elements $a\in B$ for which $\pi(a)=1$ forms a subgroup of $B$, denoted by $\ker \varphi$  and called the \emph{kernel} of $\varphi$.  It is easy to see that two elements $a,b\in B$ belong to the same  right coset of $\ker \varphi$ in $B$ if and only if $\pi(a)=\pi(b)$, and that $\varphi$ is an automorphism of $B$ if and only if $\ker\varphi = B$.
Also the following lemma provides a helpful method for determining whether an element $y \in Y$ from a skew product $G = BY$ determines an automorphism or a proper skew morphism of $B$.    This was already observed in~\cite{ConderJajcayTucker}, but we also prove it here.

\begin{lemma}
\label{lem:kernel}
Let $G=BY$ be a skew product group for the group $B$,  and let $\varphi$ be the skew morphism of $B$ given by left multiplication of $B$ by a generator $y$ of $Y$.  Then $\ker \varphi$ is the largest subgroup  of $B$ normalised by $Y$. In particular, $\varphi$ is an automorphism of $G$ if and only if $B$ is normal in $G$.
\end{lemma}
\begin{proof}
For $b\in B$, clearly $b\in \ker \varphi$ if and only if $yb=b'y$ for some $b' \in B$, which happens if and only  if $yby^{-1} = b' \in B$. This proves the first assertion. The second one follows easily from the facts that  $\varphi$ is an automorphism of $B$ if and only if $\ker \varphi = B$, and that $B$ is normal in $G = BY$  if and only if it is normalised by $Y$.
\end{proof}

More generally, the  subgroup $\ker \varphi$ is not always normal in $B$, but it is non-trivial whenever $B$ is non-trivial (see~\cite[Theorem 4.3]{ConderJajcayTucker}).
Also the restriction of $\varphi$ gives an isomorphism from $\ker \varphi$ to $\varphi(\ker \varphi)$, and hence $\varphi$ restricts to an automorphism of $\ker\varphi$ if the latter is preserved by $\varphi$. In the case where $B$ is abelian, every skew morphism $\varphi$ of $B$ induces an automorphism of $\ker\varphi$ (see~\cite[Lemma 5.1]{ConderJajcayTucker}).

An immediate consequence of non-triviality of the kernel is that every skew morphism of a cyclic group of prime order is an automorphism.  Hence to classify skew morphisms of simple groups, it suffices to consider non-abelian simple groups. Moreover, by our comments in the Introduction (explained in more detail in~\cite[Lemma 4.1]{ConderJajcayTucker}), it is sufficient for us to consider skew product groups $G = BY$ with $Y$ cyclic and core-free in $G$,  because every skew morphism of $B$ arises in this way.

\section{Proof of Theorem~\ref{thm:main}}
\label{sec:proof}

\subsection{Preliminaries}
\label{prelims}

We begin with a number of preliminary observations that are straightforward exercises, but we include their proofs for completeness.

\begin{lemma}
\label{lemma:basic}
Let $G$ be a group with subgroups $B$ and $C$ such that $G=BC,$ and $B$ is core-free and $C$ is abelian.
Then $B\cap C=\{1\}$.
\end{lemma}
\begin{proof}
Consider the natural action of $G$ on the right coset space $(G\!:\!B)$, with point-stabiliser $B$. Because $B$ is core-free, this action is faithful. But also $G=BC$, so $C$ acts transitively on $(G\!:\!B)$, and then since $C$ is abelian, it must act regularly (because the stabiliser of every point is the same), and therefore $B\cap C=\{1\}$.
\end{proof}

\begin{lemma}
\label{soceqsoc}
If $G$ has trivial soluble radical, and $\soc(G)\leq H\leq G$, then $\soc(H)=\soc(G)$, and therefore $H$ has trivial soluble radical.
\end{lemma}
\begin{proof}
Suppose to the contrary that $\soc(H) \ne \soc(G)$.  Then some minimal normal subgroup $N$ of $H$ is not contained in $\soc(G)$.  On the other hand, let $M$ be a minimal normal subgroup of $G$. Then $M\leq\soc(G)\leq H$ and so $M$ is a normal subgroup of $H$. Also $N$ is a minimal normal subgroup of $H$ but $N\nleq\soc(G)$, and so $N\cap M=\{1\}$, and therefore $N$ centralises $M$.  Since $M$ was an arbitrary minimal normal subgroup of $G$, it follows that $N$ centralises $\soc(G)$. But in a group with trivial soluble radical, the socle has trivial centraliser (see \cite[Corollary 4.3A]{DixonMortimer} for example), and so $N$ cannot centralise $\soc(G)$, contradiction.    Thus $\soc(H)=\soc(G)$, and the rest follows easily.
\end{proof}

\begin{corollary}
\label{soceqsoccor}
Let $G$ be a group with a monolithic subgroup $B$, and let $A$ be the monolith of $B$. If $B$ is not core-free in $G$, then $A$ is normal in $G$.   
\end{corollary}
\begin{proof}
Let $K$ be the core of $B$ in $G$. Then $\soc(B) = A\leq K \leq B$, and by Lemma~\ref{soceqsoc} we find that $\soc(K)=A$, making $A$ characteristic in $K$ and hence normal in $G$. 
\end{proof}

\begin{lemma}
\label{lemma:prop}
Let $G$ be a group with subgroups $B$ and $C$ such that $G=BC$ and $B\cap C=\{1\}$. If $N$ is a normal subgroup of $G$, then  $|N|=|BN\cap C||B\cap N|=|B\cap CN||C\cap N|$.
\end{lemma}
\begin{proof}
First $G=(BN)(CN)$, and $|BN| = |B||N|/|B \cap N|$ and $|CN| = |C||N|/|C \cap N|$, so
\begin{equation*}
|B||C|=|G|=|(BN)(CN)| = \frac{|BN||CN|}{|BN\cap CN|}=\frac{|B||C||N|^2}{|BN\cap CN||B\cap N||C\cap N|}
\end{equation*}
and it follows that
\begin{equation}\label{eqeqbas}
|N|^2=|BN\cap CN||B\cap N||C\cap N|.
\end{equation}
Next by Dedekind's modular law ~\cite[Item 1.3.14]{Robinson} we have $(BN\cap C)N=BN\cap CN$,  and so
$$|BN\cap CN|=|(BN\cap C)N| = \frac{|BN\cap C||N|}{|BN\cap C\cap N|}=\frac{|BN\cap C||N|}{|C\cap N|}.$$
Combining this with \eqref{eqeqbas} gives the first equality $|N|=|BN\cap C||B\cap N|$, and then exchanging the roles of $B$ and $C$  gives the second equality.
\end{proof}

Next, the following theorem proved by  Lucchini in \cite{Lucchini} will be helpful. (It was also used to prove Theorems 4.2 and 4.3 on orders and kernels of skew morphisms in \cite{ConderJajcayTucker}.)

\begin{theorem}[\cite{Lucchini}]
\label{thm:Lucc}
If  $C$ is a core-free cyclic proper subgroup of a group $G$, then $|C|<|G:C|$.
\end{theorem}

Finally, our proof of Theorem~\ref{thm:main} depends heavily on results of a 2012 paper by Li and Praeger~\cite{LiPraeger}, on finite permutation groups containing a regular cyclic subgroup.
In particular, we will frequently refer to the following theorem which we reproduce for the benefit of readers.

\begin{theorem}[\cite{LiPraeger}]
\label{thm:LiPraeger}
Let $G$ be a finite permutation group containing a regular cyclic subgroup. Then\emph{:} \\[-15pt]
\begin{enumerate}
\item[{\rm (a)}] $G$ is quasiprimitive if and only if $G$ is primitive and appears in Table~\emph{1;}
\item[{\rm (b)}]  $G$ is almost simple if and only if $G$ appears in lines $3$ to $6$ of Table~\emph{1;} and
\item[{\rm (c)}]  $G$ is imprimitive but has a transitive minimal normal subgroup $N$ if and only if  $G$ \newline and $N$ appear in Table~\emph{2.}
\end{enumerate}
\end{theorem}

\begin{table}[ht]
\label{Table1}
\renewcommand{\arraystretch}{1.1}
\begin{center}
\begin{tabular}{|ccc|ccc|clc|}
\hline
 & Degree & & & Group $G$  & & & Comments &  \\
\hline
 & $p \ge 2$ & & & $\C_p \le G \le \AGL(1,p)$ & & & $p$ prime & \\
 & $4$ & & & $\Sym(4)$ & & &  & \\
 & $n \ge 4$ & & & $\Alt(n)$ or $\Sym(n)$ & & & $n$ odd if $G = \Alt(n)$ & \\
 & $(q^{d}-1)/(q-1)$ & & & $\PGL(d,q) \le G \le \PGammaL(d,q) $ & & & $d \ge 2$, and $q$ a prime-power & \\
 & $11$ & & & $\PSL(2,11)$ or $\M_{11}$ & & &  & \\
 & $23$ & & & $\M_{23}$ & & &  & \\
\hline
\end{tabular}
${}$\\[+2pt] ${}$
\end{center}
\caption{Primitive groups with a regular cyclic subgroup}
\renewcommand{\arraystretch}{1.0}
\end{table}

${}$\vskip -36pt

\begin{table}[ht]
\label{Table2}
\renewcommand{\arraystretch}{1.1}
\begin{center}
\begin{tabular}{|ccc|ccc|ccc|}
\hline
 & Degree & & & Group $G$  & & & Subgroup $N$ &  \\
\hline
 & $15$ & & & $\Alt(5) \times \C_3$ or $(\Alt(5) \times \C_3).\C_2$ & & & $\Alt(5)$ & \\
 & $22$ & & & $\M_{11} \times \C_2$ & & & $\M_{11}$ & \\
 & $r(q^{d}-1)/(q-1)$ & & & $(\PSL(d,q) \times \C_r).\C_k$ for certain $k$ & & & $\PSL(d,q)$ & \\
\hline
\end{tabular}
${}$\\[+6pt] ${}$
\end{center}
\caption{Imprimitive groups with a regular cyclic subgroup and a transitive minimal normal subgroup}
\end{table}

\subsection{The non-core-free case}

Here we deal with the case when $B$ is not core-free in $G = BY$.

\begin{proposition}
\label{prop:notcorefree}
Let $G$ be a group with subgroups $B$ and $Y$ such that $G=BY$,
where $B$ is monolithic, and $Y$ is cyclic and core-free in $G$.
Also let $A$ be the monolith of $B,$ and let $Z$ be the centraliser of $A$ in $Y$.
If $B$ is not core-free in $G$, then $A$ is normal in $G,$ and either $B=A$ and $Z=\{1\}$,
or $|Z|< |B/A|$.
\end{proposition}
\begin{proof}
First, by Corollary~\ref{soceqsoccor} we find that $A$ is normal in $G$. Let $\overline{X}$ denote $XA/A$ ($\cong X/(X \cap A)$) for every $X \le G$. Since $G=BY$, we have $\overline{G}=\overline{B}\hskip 2pt \overline{Y}$ and thus $|\overline{G}|\leq |\overline{B}| |\overline{Y}| $. Also because the centre of $A$ must be trivial, we have $A\cap Z=\{1\}$ and therefore $\overline{Z} \cong Z$.

We now show that $\overline{Z}$ is core-free in $\overline{G}$.  For suppose that $X$ is a subgroup of $Z = \C_Y(A)$ such that $\overline{X}$ is normal in $\overline{G}$. Then $AX$ is normal in $G$, and also $AX \cong A\times X$,
where $X$ is cyclic while $A$ is a direct product of non-abelian simple groups, so $X$ is characteristic
in $AX$ and therefore normal in $G$.  But $Y$ is core-free, and so it follows that $X=\{1\}$, and hence $\overline{X} = \{1\}$.

Next let $\overline{K}$ be the core of $\overline{Y}$ in $\overline{G}$.
Since $\overline{K}$ is cyclic, all its subgroups are characteristic and hence normal in $\overline{G}$.
In particular, $\overline{Z}\cap \overline{K}$ is normal in $\overline{G}$, but then since $\overline{Z}$ is
core-free in $\overline{G}$, we find that $\overline{Z}\cap \overline{K} =\{1\}$,
and therefore $|\overline{Z}||\overline{K}|\leq |\overline{Y}|$.

Finally, if $\overline{Y}=\overline{G}$ then $Z$ is central in $AY=G$, and as $Y$ is core-free in $G$
we have $Z=\{1\}$.
On the other hand, if $\overline{Y}<\overline{G}$ then $\overline{Y}/\overline{K}$ is a proper cyclic core-free
subgroup of $\overline{G}/\overline{K}$, and so  by Theorem~\ref{thm:Lucc} we have
$|\overline{Y}/\overline{K}| < |(\overline{G}/\overline{K})/(\overline{Y}/\overline{K})| = |\overline{G}|/|\overline{Y}|$,
and therefore

\smallskip\smallskip
${}$ \hskip 100 pt $|Z|=|\overline{Z}| \leq |\overline{Y}|/|\overline{K}| < |\overline{G}|/|\overline{Y}| \leq |\overline{B}| = |B/A|$.
\end{proof}
\vskip 0pt

\subsection{The core-free case}
Here we deal with the other case, where $B$ is core-free in $G = BY$.

\begin{proposition}
\label{newprop}
Let $G$ be a group with core-free subgroups $B$ and $Y$ such that $G=BY,$ where $B$ is monolithic with monolith $A$, and $Y$ is cyclic. Then $G$ has a unique minimal normal subgroup $N$, and this normal subgroup $N$ contains $A$.
\end{proposition}
\begin{proof}
First, by Lemma~\ref{lemma:basic} we find that $B\cap Y=\{1\}$. Now let $N$ be any minimal normal subgroup of $G$ and let $\overline{X}$ denote $XN/N$ ($\cong X/(X \cap N)$), for every $X \le G$.  We will show that $N$ intersects $A$ non-trivially, and thus $N$ contains $A$ and is the unique minimal normal subgroup of $G$.

Suppose that $A\cap N=\{1\}$. Since $A$ is the unique minimal normal subgroup of $B$, this implies that $B\cap N=\{1\}$ and hence that $\overline{B}\cong B$. Moreover, as $B \cap N = B \cap YN \cap N$ this gives
$$B\cap YN \cong (B\cap YN)/(B\cap N) \cong (B\cap YN)N/N \leq YN/N\cong Y/(Y\cap N),$$
and so $B\cap YN$ is cyclic. Then by It{\^ o}'s Theorem~\cite{Ito} on products of abelian groups, $(B\cap YN)Y$ is metabelian, and therefore so is its subgroup $BY \cap YNY = BY \cap YN = G \cap YN = YN$. This implies that the minimal normal subgroup $N$ of $G$ is abelian, and hence is isomorphic to $\C_p^{\,n}$ for some prime $p$. Also $\C_B(N)$ is a normal subgroup of $B$, and because $A$ is the monolith of $B$, it follows that  either $\C_B(N)=\{1\}$ or $A\leq \C_B(N)$. We eliminate these two cases separately.

\medskip\noindent
\textbf{Case (a): $\C_B(N)=\{1\}$.}
\medskip

In this case, $B$ acts faithfully on $N$ and so is isomorphic to a subgroup of $\Aut(N)\cong \GL(n,p)$. Also $B\cap YN$ and $Y\cap N$ are cyclic, and so using Lemma~\ref{lemma:prop} we find that $B\cap YN$ has order $|B\cap YN| = |N|/|Y\cap N| \geq |N|/p = p^{n-1}$. Next, a Sylow $p$-subgroup of $\GL(n,p)$ has exponent $p^{\lceil \log_p n\rceil}$ (as shown in~\cite[Theorem~1]{Niven}, for example), and again since $B \cap YN$ is cyclic, it follows that $p^{n-1} \leq |B\cap YN|  \leq p^{\lceil \log_p n\rceil}$ and so $n-1\leq \lceil \log_p n\rceil$. This is a strong restriction, which holds only when $n\leq 2$ or $(n,p)=(3,2)$.

But if $n\leq 2$ then $n = 2$ and $p$ is odd since $B$ is insoluble, and then $[B,B]$ is isomorphic to an insoluble subgroup of $\SL(2,p)$, and hence must contain the unique involution in $\SL(2,p)$, namely $-I$, which contradicts the fact that $B$ has trivial centre.

Thus $(n,p)=(3,2)$, so $B$ is isomorphic to an insoluble subgroup of $\GL(3,2)$.  This is actually the second smallest simple group, of order $168$, and hence $B \cong \GL(3,2)$.  Also $|N| = p^n = 8$, and then since $B\cap N=\{1\}$ we find that $BN \cong N \rtimes B$ is isomorphic to the semi-direct product $\C_2^{\, 3} \rtimes \GL(3,2) \cong \AGL(3,2)$.
Moreover, by Lemma~\ref{lemma:prop} we find that $|N|=|BN\cap Y||B\cap N| = |BN\cap Y|$,
and so $BN\cap Y$ is cyclic of order $8$.  But $\AGL(3,2)$ has no element of order $8$, contradiction.

\medskip\noindent
\textbf{Case (b): $A\leq \C_B(N)$.}
\medskip

In this case, $AN$ is isomorphic to $A\times N$, and since $A$ is a direct product of non-abelian simple groups while $N$ is abelian, we see that $A$ is characteristic in $AN$.  But $A$ is not normal in $G$ (since $B$ is core-free in $G$), and so $AN$ cannot be normal in $G$, and therefore $\overline{A}$ is not normal in $\overline{G}$. Also $\overline{B}$ is monolithic with monolith $\overline{A}$, hence using Corollary~\ref{soceqsoccor} we find that $\overline{B}$ is core-free in $\overline{G}$. Now by Lemma~\ref{lemma:basic}, it follows that $\overline{B}\cap\overline{Y}=\{1\}$, so $|\overline{G}|=|\overline{B}||\overline{Y}|$, and therefore
$$|B||Y| = |G| = |\overline{G}||N| = |\overline{B}||\overline{Y}||N| = |B||YN/N||N| = |B||YN|.$$
This gives $|Y| =|NY|$, and so $N\leq Y$, which contradicts the fact that $Y$ is core-free in $G$.

\medskip

Thus $A\cap N\neq \{1\}$. But $A$ and $N$ are both normalised by $B$, so $A\cap N$ is normal in $B$,
and then since $A$ is the unique minimal normal subgroup of $B$, we find that $A\leq N$.
Hence every minimal normal subgroup of $G$ contains $A$.

Finally, any two minimal normal subgroups intersect trivially, and so there cannot be more than one,
and hence $G$ has a unique minimal normal subgroup, as required.
\end{proof}

\begin{corollary}
\label{cor:final}
Let $G$ be a group with core-free subgroups $B$ and $C$ such that $G=BC$. If $B$ is monolithic and $C$ is cyclic, then $G$ is almost simple.
\end{corollary}
\begin{proof}
We use induction on $|G|$.  By Lemma~\ref{lemma:basic}, once again we have $B\cap C=\{1\}$. Now let $A$ be the socle of $B$. Then $A$ is not normal in $G$ because $B$ is core-free in $G$, and by Proposition~\ref{newprop} we know that $G$ has a unique minimal normal subgroup $N$, and $A\leq N$. In particular, $N = \soc(G)$, and because $A$ is not soluble, $N\cong T^n$ for some non-abelian simple group $T$ and some positive integer $n\geq 1$. Also by Lemma~\ref{soceqsoc} with $H = BN$ we have  $\soc(BN)=\soc(G)=N$.

Next, since $B$ is core-free in $G$, we may view $G$ as a permutation group on the coset space $(G\!:\!B)$ with point-stabiliser $B,$ and a regular cyclic subgroup $C$. We proceed by considering three different cases, where $BN = G$, or $BN < G$ and $B$ is core-free in $BN$, or $BN < G$ and $B$ is not core-free in $BN$.

\medskip\noindent
\textbf{Case (a)}: $BN = G$
\medskip

In this case, $N$ is a transitive subgroup of $G$, and then since $N$ is the only minimal normal subgroup of $G$, it follows that every non-trivial normal subgroup is transitive, and so by definition $G$ is \emph{quasiprimitive} on $(G\!:\!B)$. Hence $G$ appears in Table~\ref{Table1} of Theorem~\ref{thm:LiPraeger}. But also $G$ is not soluble, and hence we can rule out lines $1$ and $2$ of the table, and conclude that $G$ is almost simple.

\medskip\noindent
\textbf{Case (b)}: $BN < G$ and $B$ is core-free in $BN$
\medskip

In this case, first we note that $BN$ has trivial soluble radical, since $\soc(BN)=N\cong T^n$ where $T$ is non-abelian simple, and hence the cyclic subgroup $BN\cap C$ is core-free in $BN$. Next, by Dedekind's Modular Law, $B(BN\cap C)  = BN \cap BC = BN \cap G = BN$.  We may now apply the inductive hypothesis to $BN$ (with $BN\cap C$ playing the role of $C$)   to conclude that $BN$ is almost simple, and then since $\soc(BN)=\soc(G)=N,$ it follows that also $G$ is almost simple.

\medskip\noindent
\textbf{Case (c)}: $BN < G$ and $B$ is not core-free in $BN$
\medskip

In this case, let $K$ be the core of $B$ in $BN$. Then by hypothesis $K$ is a non-trivial normal subgroup of $B$, so must contain $A$, and then by Lemma~\ref{soceqsoc} applied to $A = \soc(B) \le K \le B$, we find that $\soc(K) = \soc(B) = A$. In particular, $A$ is characteristic in $K$ and hence normal in $BN$, so $A$ is also normal in $N$.  Hence $A$ must be a product of some of the (simple) direct factors of $N$. It then follows that every subgroup of $N$ containing $A$ must be of the form $A \times H$ for some $H \le N$,  and in particular, $B \cap N$ has this form. But also $A$ is the unique minimal subgroup of $B$, so the  centraliser of $A$ in $B$ is trivial, and so $H$ is trivial, and $B \cap N = A$.

Again viewing $G$ as a permutation group on the coset space $(G\!:\!B)$ with point-stabiliser $G_v = B$, with regular cyclic subgroup $C$, we now have $A = B \cap N = N_v$.  Since this is normal in $N$, it follows that $N$ acts regularly on each of its orbits.

Let $\Delta$ be an $N$-orbit on $(G\!:\!B)$, and let $G_{\Delta}$ denote the set-wise stabiliser of $\Delta$ in $G$. Then since $N$ is normal in $G$ and preserves $\Delta$, we know that also $N$ is normal in $G_{\Delta}$, and hence $N^{\Delta}$ is normal in  $G_{\Delta}^{\,\Delta}$.
By~\cite[Theorem 1.3(3)]{LiPraeger} it now follows that $G_{\Delta}^{\,\Delta}$ appears either in lines 3 to 6 of Table 1 or in lines 1 to 3 of Table 2 of Theorem~\ref{thm:LiPraeger}, and that $N^{\Delta}$ is a simple subgroup of $G_{\Delta}^{\,\Delta}$ acting regularly on $\Delta$. This is not possible, however, because for each choice of $G_{\Delta}^{\,\Delta}$ from those two tables, the degree of the permutation group is strictly smaller than the order of its smallest non-abelian simple normal subgroup.
Hence this third case is eliminated.
\end{proof}

We can now complete the proof of Theorem~\ref{thm:main}, below.

\begin{proof}
Let $G$ be a group with core-free subgroups $B$ and $Y$ such that $G=BY$, where $B$ is monolithic and $Y$ is cyclic. By Lemma~\ref{lemma:basic}, we know that $B\cap Y=\{1\}$, and Corollary~\ref{cor:final} implies that  $G$ is almost simple. Once again, we may view $G$ as a permutation group on the coset space $(G\!:\!B)$ with point-stabiliser $B$, and a regular cyclic subgroup $Y$. It follows that $G$ appears in lines 3 to 6 of Table 1 in Theorem~\ref{thm:LiPraeger}. But also we can exclude line $4$ of the table, because in that case the point-stabiliser  contains an elementary abelian normal subgroup of order $q^{d-1}$, and hence is not monolithic. The other possibilities are genuine examples, and give the cases listed in the statement of Theorem~\ref{thm:main}.
\end{proof}

\section{Examples and remarks concerning Theorem~\ref{thm:main}}
\label{sec:remarks}

Our first example in this section gives an infinite family for which the upper bound on the order of $Z = \C_Y(A)$ in Theorem~\ref{thm:main} is sharp.

\begin{example}
\label{example:sharp}
Let $p$ be a prime, let $G = \PSigmaL(2,2^p)\times \AGL(1,p)$ and write $\PSigmaL(2,2^p)=A\rtimes C\cong \PSL(2,2^p)\rtimes \C_p$ and $\AGL(1,p)=T\rtimes M\cong \C_p \rtimes \C_{p-1}$. Let $D$ be the diagonal subgroup of $C\times T\cong\C_p \times \C_p$, let $B=\langle A,D\rangle$ and let $Y=C\times M\cong \C_p \times \C_{p-1}$.
Then $B$ is isomorphic to $\PSigmaL(2,2^p)$, and hence is almost simple with socle $A$, which is normal in $G$.  Also $Y$ is cyclic, $B\cap Y=\{1\}$ and, by order considerations, $G=BY$. Note also that $Y$ is core-free in $G$. Thus $B$, $Y$ and $G$ satisfy the hypotheses of Theorem~\ref{thm:main}. Finally, we see that the centraliser of $A$ in $Y$ is $M\cong\C_{p-1}$, while $|B/A|=|C|=p$.
\end{example}

Taking $p \ge 3$ in Example~\ref{example:sharp} also provides examples for which $B$ is neither core-free nor normal in $G$.

\smallskip
The next example shows that although the order of $Z = \C_Y(A)$ is always less than $|B/A|$, the order of the centraliser of $A$ in $G$ can be arbitrarily large, even when $B=A$.

\begin{example}
\label{example:sharp2}
For $n\geq 5$, take $A = B = \Alt(n)$, let $g$ be an element of odd order $m$ in $A$, and let $G$ be the  direct product $A \times \C_m$.  Also let $Y$ be the cyclic subgroup of $G$ of order $m$ generated by $gc$,  where $c$ is a generator of the direct factor $\C_m$.  Then clearly $G = BY$ with $B \cap Y = \{1\}$, and $B$ is simple  and normal in $G$, while $Y$ is cyclic and core-free in $G$, so again $B$, $Y$ and $G$ satisfy the hypotheses  of Theorem~\ref{thm:main}.  Here $Z = \C_Y(A)$ is trivial, while $\C_G(A)$ has order $m$, which can be arbitrarily large (depending on the choice of $n$ and $g$).
\end{example}

Finally, we give an example where $B$ is neither core-free nor normal in $G$, and $A$ is not simple.

\begin{example}
\label{example:monolithic}
Let $H$ be a permutation group of degree $n$ with a transitive non-normal subgroup $X$ and a core-free cyclic subgroup $Y$ such that $H=XY$ with $X \cap Y = \{1\}$.

Note that there are many possibilities for $(H,X,Y)$. The smallest example in terms of  $|H|$  is obtained by taking $H=\Sym(3) \times \C_3 \cong (\C_3 \rtimes \C_2) \times \C_3$ with $X \cong \C_2 \times \C_3 \cong \C_6$  and $Y$ the diagonal subgroup of $\C_3 \times \C_3$, and $n = 6$. The smallest example in terms of the degree $n$ is obtained by taking $H=\Sym(4)$, with $X=\D_4$ and $Y \cong \C_3$, and $n=4$.

Next, let $T$ be a non-abelian simple group, let $G$ be the wreath product $T \wr H$ (which is isomorphic to $T^n \rtimes H$), and let $B =T \wr X$, and consider $Y$ as a subgroup of $\{1\} \wr H$ in $G$.  Then $A = T^n$ is a minimal normal subgroup of $B$, since $X$ acts transitively on the $n$ copies of $T$,  and in fact $T^n$ is the only minimal normal subgroup of $B$, because any other minimal normal subgroup of $B$  would be contained in the centraliser of $T^n$ in $B$, which is trivial. It follows that $B$ is monolithic, with monolith $A$. Also $G=BY$ with $B \cap Y = \{1\}$, and $B$ not normal in $G$, while $Y$ is cyclic and core-free in $G$.
\end{example}

We complete this section with a remark that includes an alternative proof of Corollary~\ref{cor:final}  in the case where $B$ is simple.

\begin{remark}
The proofs of the preliminaries in Section \ref{prelims} and Propositions~\ref{prop:notcorefree} and~\ref{newprop} are elementary, in the sense of not relying on the CFSG (the classification of finite simple groups). On the other hand, the proof of Corollary~\ref{cor:final} relies on the CFSG many times, when we cite Theorem~\ref{thm:LiPraeger}.  This, however, can be avoided in the case when $B$ is simple.  We present a short proof of this fact, as we believe it could be of some interest.

\smallskip
\begin{proof}
Let $G$ be a group with core-free subgroups $B$ and $Y$ such that $G=BY$, where $B$ is non-abelian simple, and $Y$ is cyclic.  Then $\soc(B)=B$ since $B$ is simple, and by Proposition~\ref{newprop} we know that $G$ has a unique minimal normal subgroup, say $N,$ containing $B$. Also $N \cong T^k$ for some simple group $T$, and since $B\leq N$, we have $YN=G$, and therefore $Y$ conjugates transitively the $n$ direct factors of $N$  among themselves.  Also because $Y$ is abelian, it is an easy exercise to show that $Y \cap N$ is isomorphic to a subgroup of $T$.  Hence by Theorem~\ref{thm:Lucc}, we find that $|Y \cap N| < |T:(Y \cap N)|$ and  therefore $|Y \cap N| < |T|^{1/2}$. Next, Lemma~\ref{lemma:prop} gives  $|N|=|B \cap YN||Y \cap N| = |B||Y \cap N|$,  and since $B\leq N \cong T^k$ and $B$ is simple, it can be easily verified that $B$ is isomorphic to a subgroup of $T$, and therefore $|B| \leq |T|$. Consequently $|T|^k = |N| = |B||Y \cap N| < |B||T|^{1/2} \leq |T|^{3/2}$, which gives $k<\frac{3}{2}$. It follows that $k=1$, and therefore $N$ is simple, and $G$ is almost simple.   
\end{proof}

\smallskip\smallskip
Finally, note that the last part of our proof of Theorem~\ref{thm:main} (after Corollary~\ref{cor:final}) again relies  heavily on the CFSG through the use of Theorem~\ref{thm:LiPraeger}, even in the case where $B$ is simple.
\end{remark}

\section{Skew morphisms of non-abelian simple and other monolithic groups}
\label{sec:skew}
In this section, we first describe a method for determining the skew morphisms of a finite group $B$ using skew product groups. Using Theorem \ref{thm:main}, we then apply this method to the case where $B$ is monolithic and core-free in its skew product group. We complete the paper by giving a summary of information about the case where $B$ is simple, and making some observations in the case where $B$ is not core-free.

\subsection{Determining skew morphisms using skew product groups}\label{sub:determining skews}
We first define a few notions that will be very useful. If $\varphi$ is a skew morphism of a group $B$ and $\psi$ is a group automorphism of $B$, then $\psi \circ \varphi \circ \psi^{-1}$ is also a skew morphism of $B$ (see \cite{BachratyJajcay} for example) that is said to be \emph{equivalent} to $\varphi$. Now let $G$ be a skew product group for $B$. By definition, $G$ contains a copy of $B$ and an element $y$ such that $G=B\langle y\rangle$ is a complementary factorisation and $\langle y \rangle$ is core-free in $G$. We call such a  pair $(B,y)$ a \emph{skew generating pair} for $G$, and we say that two skew generating pairs are \emph{equivalent} if they are conjugate under $\Aut(G)$. It is easy to see that equivalent skew generating pairs induce equivalent skew morphisms.

To enumerate all skew morphisms of $B$, we first determine all equivalence classes and then determine the size of each class.   To determine the classes, we first determine (up to isomorphism) all skew product groups $G$ for $B$. Note that if we want only proper skew morphisms of $B$, then by Lemma~\ref{lem:kernel} we can restrict ourselves to the skew products groups in which $B$ is not normal. Then for each such $G$, we determine the equivalence classes of generating pairs, choose a representative from each class, and take the corresponding skew morphisms. By the previous paragraph, we now have at least one representative of each equivalence class of skew morphisms. It remains to check if any of these actually represent the same class, which can be done by a direct check of conjugacy under $\Aut(B)$ in $\Sym(B)$. Finally, the size of the equivalence class of a given skew morphism $\varphi$ is given by the index $|\Aut(B) \!:\! \C_{\Aut(B)}(\varphi)|$.

The method above is designed to minimise the amount of calculations that have to be undertaken in  $\Sym(B)$, noting that this is often a very large group compared to $\Aut(B)$ or even $\Aut(G)$.

\begin{remark}
\label{rem:skewgenpairs}
If $(B,y)$ is a generating pair for a skew product group $G$, then no two different elements of $\langle y\rangle$ induce the same skew morphism of $B$. For if $y$ and $y'$ are elements of $Y$ that induce the same skew morphism $\varphi$ of $B$, with power function $\pi$, then for every $b \in B$ we have $yb = \varphi(b)y^{\pi(b)}$ and $y'b = \varphi(b)(y')^{\pi(b)},$  and so $b^{-1}y^{-1}y'b = (yb)^{-1}y'b = (\varphi(b)y^{\pi(b)})^{-1}(\varphi(b)(y')^{\pi(b)}) = y^{-\pi(b)}(y')^{\pi(b)} = (y^{-1}y')^{\pi(b)}$.   This implies that the cyclic subgroup of $\langle y\rangle$ generated by $y^{-1}y'$  is normalised by $b$, for all $b \in B$, and hence is normal in $B\langle y\rangle = G$.  But $\langle y\rangle$ is core-free in $G$,  so $y^{-1}y'$ must be trivial, and therefore $y' = y$.
\end{remark}

Next, we note the following, obtainable from the proof of \cite[Theorem~2]{JajcaySiran}:

\begin{proposition}
\label{prop:mapskews}
A skew morphism $\varphi$ of a finite group $B$ gives rise to a regular Cayley map  for $B$ if and only if the set of elements in some cycle of $\varphi$  \emph{(}when regarded as a permutation of $B)$ is closed under taking inverses and generates $B$.
\end{proposition}

Finally we note that every regular Cayley map for the group $B$ that is \emph{balanced} (in the sense defined in~\cite{JajcaySiran}) comes from an automorphism of $B$, and so proper skew morphisms give rise only to non-balanced regular Cayley maps.

\subsection{Skew morphisms of monolithic groups: the core-free case}\label{sub:corefree}
Here we apply the method from the previous subsection to determine all skew morphisms of a monolithic group $B$ 
when $B$ is core-free  in the skew product group $G=BY$.  All such skew morphisms will be proper, since the hypothesis requires  that $B$ itself is not normal in $G$.

By Theorem~\ref{thm:main}, the only possibilities for $B$ are $\Alt(5)$, $\M_{10}$, $\M_{22}$,  $\Alt(n)$ for even $n\geq 6$, and $\Sym(n)$ for $n\geq 5$.  We enumerate both the number of (proper) skew morphisms, and the number of equivalence classes of these.  This enumeration splits into seven cases, coming from the five cases in Theorem~\ref{thm:main},  with the two additional cases $\Alt(6)$ and $\Sym(6)$ treated separately because of their exceptional  outer automorphisms.  In particular, we have five sporadic cases (for which we rely heavily on {\sc Magma} \cite{Magma}) and two infinite families.


\medskip\smallskip\noindent
\textbf{Case 1: $B=\Alt(5)$}
\medskip

Here the skew product group $G$ is $\PSL(2,11)$, with $B=\Alt(5)$ core-free in $G$,
and $Y$ cyclic of order $11$. 
There are two conjugacy classes of subgroups isomorphic to $B$ in $G$, but these form a single class within $\Aut(G)$, so we may take $B$ as any representative of just one of them. Then the subgroup $H$ of $\Aut(G)$ preserving $B$ has two orbits on elements of order $11$ in $G$, with $y$ and $y^{-1}$ always lying in different orbits, for every such $y$. Hence we have two equivalence classes of skew generating pairs for $G$.

Next, take any $y$ of order $11$ in $G$, and let $\varphi$ be the skew morphism of $B$ induced by $y$.  Then every proper skew morphism of $B=\Alt(5)$ is equivalent to $\varphi$ or $\varphi^{-1}$.  Moreover, the centraliser in $\Aut(B) \cong \Sym(5)$ of $\varphi$ is trivial, and $\varphi^{-1}$ is not a conjugate of $\varphi$ under an element of $\Aut(B)$,  so these two skew morphisms are inequivalent.  Hence there are two equivalence classes of proper skew morphisms of $\Alt(5)$, each of size $120$.

Moreover, it is easy to check that $\varphi$ induces at least one $11$-cycle on $\Alt(5)$
that contains three involutions, four $3$-cycles and four $5$-cycles, forming a set that is closed
under inverses and clearly generates $\Alt(5)$.  It follows that the same holds for $\varphi^{-1}$,
and hence for all skew morphisms (each of which is equivalent to $\varphi$ or $\varphi^{-1}$).
Hence by Proposition~\ref{prop:mapskews}, every proper skew morphism of $\Alt(5)$
gives rise to a (non-balanced) regular Cayley map.

In conclusion, $\Alt(5)$ has $240$ proper skew morphisms, which split into two equivalence classes
of size $120$, with one of the classes consisting of the inverses of the members of the other one,
and also every proper skew morphism gives rise to a (non-balanced) regular Cayley map.

\medskip\smallskip\noindent
\textbf{Case 2:} $B = \M_{10}$
\medskip

In this case the skew product group $G$ is $\M_{11}$, with $B = \M_{10}$ 
and $Y$ cyclic of order $11$. 
Again here we have two classes of skew generating pairs, and the skew morphisms associated
with their representatives are not equivalent. We find there are two equivalence classes of proper skew morphisms,
each of size $|\Aut(\M_{10})| = 1440$, giving us a total of $2880$ proper skew morphisms.
Also one of the classes consists of the inverses of the members of the other one.
Finally,  each of the two class representatives (and hence every proper skew morphism) induces at least one
$11$-cycle on $\M_{10}$ that consists of three involutions, six elements of order $4$ and two of order $8$,
forming a generating set that is closed under inverses, and hence every one
of the proper skew morphisms of $\M_{10}$ gives rise to a (non-balanced) regular Cayley map.

\medskip\smallskip\noindent
\textbf{Case 3: $B = \M_{22}$}
\medskip

The skew product group $G$ in this case is $\M_{23}$, with $B = \M_{22}$ 
and $Y$ cyclic of order $23$. 
Again there are two classes of skew generating pairs, and the skew morphisms associated with
representatives of different classes are not equivalent.  Hence we have two equivalence classes of
proper skew morphisms, each of size $|\Aut(\M_{12})| = 887040$, giving us a total of $1774080$.
Again in this case, each class consists of the inverses of the members of the other one.
Finally,  each of the two class representatives (and hence every proper skew morphism) induces at least one
$23$-cycle on $\M_{22}$ that consists of seven involutions, eight elements of order $7$ and eight of order $11$,
forming a generating set that is closed under inverses, and hence every one
of the proper skew morphisms of $\M_{22}$ gives rise to a (non-balanced) regular Cayley map.

\medskip\smallskip\noindent
\textbf{Case 4: $B = \Alt(6)$}
\medskip

Here the skew product group $G$ is $\Alt(7)$, with $B = \Alt(6)$ 
and $Y$ cyclic of order $7$. 
This time there is only one equivalence class of skew generating pairs, giving us a single class of proper skew morphisms, of size $|\Aut(\Alt(6))|=1440$. Moreover, each proper skew morphism induces at least one $7$-cycle on $\Alt(6)$ that contains three involutions and four elements of order $5$, forming a generating set that is closed under inverses, and hence every proper skew morphism of $\Alt(6)$ gives rise to a (non-balanced) regular Cayley map.

\medskip\smallskip\noindent
\textbf{Case 5: $B = \Sym(6)$}
\medskip

The skew product group $G$ in this case is $\Sym(7)$, with $B = \Sym(6)$ 
and $Y$ cyclic of order~$7$.  Just as in the previous case, there is only one equivalence class of skew generating pairs, giving a single class of proper skew morphisms, of size $|\Aut(\Sym(6))|=1440$.
Moreover, a representative of that class induces at least one $7$-cycle on $\Sym(6)$ that contains five involutions and two elements of order $6$ which form a generating set that is closed under inverses, and hence every proper skew morphism of $\Sym(6)$ gives rise to a (non-balanced) regular Cayley map.

\medskip\smallskip\noindent
\textbf{Case 6:} $B = \Alt(n)$ for even $n \geq 8$
\medskip

Here the skew product group $G$ is $\Alt(n+1)$, with $B = \Alt(n)$ 
and $Y$ cyclic of order $n+1$.  
Up to equivalence, there is only one skew generating pair, 
and hence a single equivalence class of proper skew morphisms of $\Alt(n)$, of size at most $|\Aut(\Alt(n))|=n!$.

Now take $B = \Alt(n)$ as the stabiliser of the point $n+1$ in $G$, let $y = (1,2,3,\dots,n+1)$, and let $\varphi$ be the skew morphism of $\Alt(n)$ induced by $y$. If $\pi$ is the power function of $\varphi$, then by considering the effect  of $\,y b=\varphi(b)y^{\pi(b)} \in G = \Alt(n+1)\,$ on the point $n+1$, we find that
$$1^b = ((n+1)^{y})^b = (n+1)^{y b} = (n+1)^{\varphi(b)y^{\pi(b)}} = (n+1)^{y^{\pi(b)}}  \ \hbox{ for all } b \in B, $$
because $\varphi(b) \in B$ fixes the point $n+1$. Hence if $\pi(b)=k$ then $1^b = (n+1)^{y_1^{\,k}} = k$, and therefore $\pi(b) = 1^b$ for all $b \in B$.

Next, suppose $z$ is another $(n+1)$-cycle in $G = \Alt(n+1)$ that gives exactly the same skew morphism $\varphi$ of $\Alt(n)$, say $z = (p_1,p_2,p_3,\dots,p_n,n+1)$. Then the same argument applies also to $\,z b=\varphi(b)z^{\pi(b)}$, giving $((n+1)^{z})^b = (n+1)^{z^{\pi(b)}}$ for all $b \in B$, so $p_1^{\, b} = ((n+1)^{z})^b = (n+1)^{z^{\,k}} = p_k$ whenever  $\pi(b)=k$.

In particular, if we let $a$ be the $(n-1)$-cycle $(1,2,3,\dots,n-1)$ in $B = \Alt(n)$, then taking $b = a^j$ we find that $\pi(a^j) = 1^{a^j} = j+1$ and so $p_1^{\,a^j} = p_{j+1}$ for $1 \le k \le n-1$.  It follows that $a$ takes $p_i$ to $p_{i+1}$ for $1 \le i \le n-1$, and so has the same effect as $z$ on $\{p_1,p_2,\dots,p_{n-1}\}$, and therefore $z = (1,2,3,\dots,n,n+1) = y$.

Hence for a given copy of  $B = \Alt(n)$ in $G = \Alt(n+1)$, every two different $(n+1)$-cycles give two different skew morphisms of $B$. Since the number of $(n+1)$-cycles is $n!$, it follows that the number of proper skew morphisms of $\Alt(n)$ is exactly $n!$.  In particular, the centraliser in $\Aut(B) \cong \Sym(n)$ of $\varphi$ is trivial.

Also the set $S$ of elements of the cycle of $\varphi$ on $\Alt(n)$ containing the double transposition $x = (1,2)(3,4)$ consists of the following elements:
\\[-10pt]
\begin{center}
\begin{tabular}{rcl}
$x$ & \hskip -5pt $=$ &  \hskip -4pt $(1,2)(3,4)$, \\[+2pt]
$yxy^{-2} $ & \hskip -5pt $=$ &  \hskip -4pt $(1, n, n-1, n-2, \dots, 5, 4, 3)$, \\[+2pt]
$y^{2}xy^{-1} $ & \hskip -5pt $=$ &  \hskip -4pt $(1, 3, 4, 5, \dots, n-2, n-1, n)$, \\[+2pt]
$y^{3}xy^{-4} $ & \hskip -5pt $=$ &  \hskip -4pt $(1, n, n-2, n-3, \dots, 4, 3, 2)$, \\[+2pt]
$y^{4}xy^{-3} $ & \hskip -5pt $=$ &  \hskip -4pt $(1, 2, 3, 4, \dots, n-3, n-2, n)$, \\[+2pt]
$y^{5}xy^{-5} $ & \hskip -5pt $=$ &  \hskip -4pt $(n-3, n-2)(n-1, n)$, \\[+2pt]
$y^{6}xy^{-6} $ & \hskip -5pt $=$ &  \hskip -4pt $(n-4, n-3)(n-2, n-1)$, \\[+2pt]
  & \hskip -5pt $:$ &  \hskip -4pt ${}$ \\[+2pt]
$y^{n-1}xy^{-(n-1)} $ & \hskip -5pt $=$ &  \hskip -4pt $(3,4)(5,6)$, \\[+2pt]
$y^{n}xy^{-n} $ & \hskip -5pt $=$ &  \hskip -4pt $(2,3)(4,5)$.  \\[+2pt]
\end{tabular}
\end{center}
${}$\\[-4pt]
\indent
It is easy to see that this set $S$ generates a transitive subgroup of $\Alt(n)$, and contains the $5$-cycle $xy^{n}xy^{-n} = (1,3,5,4,2)$.   This $5$-cycle can be used to prove that the subgroup generated by $S$ is primitive,  and hence equal to $\Alt(n)$. For if $V$ is a block of imprimitivity for the subgroup, containing the point $1$,  then $V$ is preserved by the $5$-cycle $xy^{n}xy^{-n}$, and hence contains  the points $2$, $3$, $4$ and $5$.  Then furthermore, $V$ must be preserved by the $(n-1)$-cycle  $y^{2}xy^{-1} = (1, 3, 4, 5, \dots, n-2, n-1, n)$, and so contains all of the points $1$ to $n$.  Hence $\langle S \rangle$ is primitive.  Finally by Jordan's Theorem~\cite[Theorem 3.3E]{DixonMortimer}, the presence of the $5$-cycle also implies that  $\langle S \rangle = \Alt(n)$.

As also $S$ is closed under taking inverses, it follows that $\varphi$ and hence every proper skew  morphism of $\Alt(n)$ gives rise to a (non-balanced) regular Cayley map.

\medskip
\textbf{Case 7:} $B=\Sym(n)$ for $n=5$ and $n \geq 7$
\medskip

This case, where the skew product group $G$ is $\Sym(n+1)$, and $B = \Sym(n)$ and $Y$ cyclic of order $n+1$, is similar to the previous case, with only one equivalence class of skew generating pairs and a single equivalence class of proper skew morphisms. Again the size of this equivalence class is $|\Aut(B)|=n!$.

Also if we take $y$ as the $(n+1)$-cycle $(1,2,3,\dots,n,n+1)$, and $\varphi$ as the skew morphism  of $\Sym(n)$ induced by $y$, then the set $T$ of elements of the cycle of $\varphi$ on $\Sym(n)$ containing  the single transposition $(1,2)$ consists of the $n$-cycle $(1, 2, 3, 4, \dots, n-2, n-2, n)$ and its inverse,  plus the $n-1$ transpositions $(1,2)$, $(2,3)$, $\dots$, $(n-2,n-1)$, $(n-1,n)$.   The first $n$-cycle and any one of these transpositions generate $\Sym(n)$, and clearly $T$ is closed  under inverses, so again it follows that $\varphi$ and hence every proper skew morphism of $\Sym(n)$  gives rise to a (non-balanced) regular Cayley map.

\medskip

\begin{remark}
\label{rem:skewregcaymap}
Note that all skew morphisms considered in Cases 1 to 7 above induce a cycle which forms a generating set and is closed under inverses. Hence by Proposition~\ref{prop:mapskews} they all give rise to a regular Cayley map.

\end{remark}

In all of the cases above the centraliser in $\Aut(B)$ of $\varphi$ is trivial. Consequently, this is true for all proper skew morphisms of simple groups. On the other hand, there are many examples of groups which admit a skew morphism with non-trivial centraliser.  The latter is true even for some almost simple groups (see Example~\ref{example:sym5}).




\subsection{Summary of skew morphisms of simple groups}\label{sec:skewsimple}
Let $G=BY$ be the skew product group corresponding to some skew morphism of a simple group $B$. If $B$ is not core-free in $G$, then it is normal in $G$ and, by  Lemma \ref{lem:kernel}, the skew morphism is an  automorphisms of $B$. If $B$ is core-free in $G$, then by Theorem~\ref{thm:main}, $B$ is one of   $\Alt(5)$, $\M_{22}$ or $\Alt(n)$ for even $n\geq 6$.  In particular, this proves Corollary~\ref{cor:main}.

Moreover, as we have shown in Section~\ref{sub:corefree},  $\Alt(5)$, $\Alt(6)$ and $\M_{22}$ admit $240$, $1440$ and $1774080$
proper skew morphisms, respectively while, for even $n \ge 8$, $\Alt(n)$ admits $n!$ proper skew morphisms.
Skew morphisms of $\Alt(5)$ and $\M_{22}$ fall into two equivalence classes of equal size,
while those of $\Alt(n)$ form a single equivalence class, for even $n \ge 6$.

By Remark \ref{rem:skewregcaymap} we also have the following.

\begin{theorem}
Every proper skew morphism of a non-abelian finite simple group $B$ gives rise to a non-balanced regular Cayley map for $B$.
\end{theorem}

Unfortunately we cannot say much about the regular Cayley maps that arise from automorphisms of non-abelian finite simple groups (or monolithic groups in general), because that requires  much greater knowledge about generating sets for such groups.

\subsection{Skew morphisms of monolithic groups: the non-core-free case}

We now consider the case where the monolithic group $B$ is not core-free in the skew product group $G=BY$. If $B$ is simple, then it is normal in $G$, and hence all skew morphisms will be automorphisms of $B$. In contrast, if $B$ is monolithic but not simple, it might not be normal in $G$ (see Example~\ref{example:sharp}) and hence the corresponding skew-morphism might be proper. Moreover, if $B$ is almost simple, the corresponding skew morphism always restricts to an automorphism of $\soc(B)$; see Corollary~\ref{restrictsTo}.

\begin{lemma}
\label{lem:uniquesocle}
If $B$ is an almost simple group with socle $A$, then $A$ is the only subgroup of $B$ isomorphic to $A$.
\end{lemma}
\begin{proof}
Let $T$ be a subgroup of $B$ isomorphic to $A$.  Then $A \cap T$ is normal in $T$ since $A$ is normal in $B$,  but $T$ is simple, so either $A \cap T = T$ or $A \cap T =\{1\}$.  In the latter case, $T$ is isomorphic to a subgroup of $B/A$, which in turn is a subgroup of $\Out(A)$  since $A$ is the socle of the almost simple group $B$, but that is impossible since $T$ is insoluble while $\Out(A)$ is soluble, by the proof of Schreier's Conjecture (see~\cite[Section~4.7]{DixonMortimer}). Thus $A \cap T = T$, and so $T = A$.
\end{proof}

\begin{corollary}\label{restrictsTo}
Let $B$ be an almost simple group with socle $A$, and let $\varphi$ be a skew morphism of $B$ associated with the skew product group $G=BY$. If $B$ is not core-free in $G$, then  $\varphi$ restricts to an automorphism of $A$.
\end{corollary}
\begin{proof}
By Proposition~\ref{prop:notcorefree}, $A$ is normal in $G$,  and so $\ker\varphi$ contains $A$ by Lemma~\ref{lem:kernel}.   It follows that $\varphi$ restricts to an isomorphism from $A$ to the subgroup $T = \varphi(A)$ of $B$,  but then $A=T$ by Lemma \ref{lem:uniquesocle}, and therefore $\varphi$ restricts to an automorphism of $A$.  
\end{proof}

Note here that if $B$ is almost simple with socle $A$, then $A$ must be a relatively large subgroup of $B$, because $|B : A| \leq |\Out(A)| < \log_2 |A|$; see~\cite{Kohl} for example.

\medskip

It would be interesting to obtain a generalisation of Lemma \ref{lem:uniquesocle} to the monolithic case, with a corresponding generalisation of Corollary~\ref{restrictsTo}.

\subsection{Skew morphisms of monolithic groups with socle of index two}\label{sec:newnew}
We now show how  Proposition~\ref{prop:notcorefree} can be used to find all skew product groups $G=BY$ for a monolithic group $B$ with monolith $A$ in  the case when $B \cong \Aut(A)$ and is not core-free in $G$, and $|B : A| = 2$.

\begin{proposition}\label{prop:index2}
Let $B$ be a monolithic group with monolith $A$, and let $G=BY$ be a skew product group for $B$ with $B$  not core-free in $G$. If $B$ is isomorphic to $\Aut(A)$ and $|B : A| =2$, then $G = \C_G(A) \rtimes B$. Moreover, $|\C_G(A)|=|Y|$  and is bounded above by the maximum order of an element in $B$, and $\C_G(A)$ has a cyclic subgroup of index at most two.
\end{proposition}
\begin{proof}
By Proposition~\ref{prop:notcorefree}, $A$ is normal in $G$ and $\C_Y(A)=1$. Since $B \cong \Aut(A)$, we have  $G/\C_G(A) \cong \Aut(A)$, and then since $B$ is monolithic, $\C_B(A)=1$ and so $G  = \C_G(A) \rtimes B$. In particular, by order considerations, we find that $|\C_G(A)| = |Y|$.  But also $Y$ acts faithfully by conjugation on $A$ since $\C_Y(A)=1$, and thus $Y$ embeds in $\Aut(A)\cong B$. As $Y$ is cyclic, it follows that $|\C_G(A)| = |Y|$ is bounded above by the maximum order of an element in $B$.

Next, write $\overline{X}$ for $XA/A$ when $X$ is a subgroup of $G$. Since $A\cap Y$ and $|B : A| = 2$, we see that $\overline{Y}\cong Y$ and $|\overline{B}|=2$. It follows that $|\overline{G}| = 2|\overline{Y}|$, and hence $\overline{Y}$ is a cyclic subgroup of index two in $\overline{G}$. This implies that  $\overline{\C_G(A)}\cap \overline{Y}$ is a cyclic subgroup of $\overline{\C_G(A)}$ of index at most $2$. Finally, since $A$ is the monolith of $B$, we find that $A \cap \C_G(A)=1$, and thus $\overline{\C_G(A)} \cong \C_G(A)$. 
\end{proof}

There are many interesting examples of a monolithic group $B$ with monolith $A$ such that $B \cong \Aut(A)$ and $|B : A| =2$, such as the following:
\begin{itemize}
\item $\Sym(n)$ for $n=5$ and $n\geq 7$,
\item $\PGL(2,p)$ for an odd prime $p$,
\item $\PGL(n,p)$ for $n\geq 3$ and a prime $p$ such that $\gcd(n,p-1) = 1$,
\item $T\wr \C_2$, where $T$ is a non-abelian simple group such that $T \cong \Aut(T)$.
\end{itemize}

\smallskip
We now explain how to find all skew product groups $G$ for such a monolithic group $B$. By the results of Section~\ref{sub:corefree}, it suffices to consider the case where $B$ is not core-free in $G$, and apply Proposition~\ref{prop:index2}. This gives us a good upper bound on the order of $\C_G(A)$. Also the groups having a cyclic subgroup of index at most two are very well-understood (see \cite{CzisterDomokos} for example),  and hence we can easily identify all possible candidates for $\C_G(A)$. For each candidate for $\C_G(A)$, we  find all semi-direct products of the form $\C_G(A) \rtimes B$. As $A$ acts trivially on $\C_G(A)$ by conjugation and $|B : A| = 2$, conjugation by $B$ induces a  group of automorphisms of $\C_G(A)$ of order at most two. 
This gives us a complete list of candidates for $G$. As a final step, we reject a candidate $G$ if it does not admit a cyclic core-free complement of $B$.

Once we have all possible skew product groups for $B$, we can attempt to find all skew morphisms using the method outlined in Section~\ref{sub:determining skews}. We now illustrate this method in the smallest relevant case, namely $B=\Sym(5)$.

\begin{example}\label{example:sym5}
Let $B=\Sym(5)$, and $A=\soc(B)=\Alt(5)$, and let $G=BY$ be a skew product group corresponding to some skew morphism of $B$. If $B$ is normal in $G$, then by Lemma~\ref{lem:kernel}, the skew morphism is an automorphism of $B$.  Also the case when $B$ is core-free in $G$ was dealt with in Section~\ref{sub:corefree}, 
and hence we may assume that the core of $B$ in $G$ is $A$. 

By Proposition~\ref{prop:index2}, we find that $G = \C_G(A) \rtimes B$, and that $|\C_G(A)|$ is bounded above by the maximum order of an element in $B$, namely $6$. Then since $G = \C_G(A) \rtimes B$ where $B$ is not normal in $G$, it follows that $|\C_G(A)|\geq 3$. In particular, $\C_G(A)$ is isomorphic to $\C_3$, $\C_4$, $\C_2^{\, 2}$, $\C_5$, $\C_6$ or $\Sym(3)$. 

It remains to determine the conjugation action of $B$ on $\C_G(A)$.  Since $B$ is not normal in $G$, this action is non-trivial, but its index $2$ subgroup $A$ acts trivially by definition, and so $B$ induces an automorphism of $\C_G(A)$ of order $2$. In each case, there is a unique conjugacy class of element of order $2$ in $\Aut(\C_G(A))$, and hence a unique possibility for $G$. 

Now that we know $G$, it remains to find $Y$, which must be a cyclic core-free complement for $B$ in $G$. It turns out that $G$ has no such subgroup when  $\C_G(A) \cong \C_4$ or $\C_G(A) \cong \C_6$, and a unique class of such subgroups in all other cases, up to conjugacy under $\N_G(B)$. (This is a slightly tedious computation to do by hand, but it is easy to do by computer.)

Using similar methods as in Section~\ref{sub:corefree}, we enumerate all skew morphisms for these skew product groups. When $\C_G(A)$ is isomorphic to $\C_3$,  $\C_2^{\, 2}$, $\C_5$ or $\Sym(3)$,  we get a single equivalence class  of skew morphisms in each case, containing $20$, $30$, $24$ and $20$ proper skew morphisms, respectively. Indeed in all of these cases, the centraliser in $\Aut(B)$ of the skew morphism is non-trivial. 

Thus $\Sym(5)$ admits $20+30+24+20 = 94$ proper skew morphisms (in four equivalence classes) when $\Alt(5)$ is the core of $\Sym(5)$ in its skew product group, plus another $120$ proper skew morphisms in a single equivalence class when the core is trivial, and $120$ automorphisms when the core is $\Sym(5)$.
\end{example}

Using a computer to automate the approach used in Example~\ref{example:sym5}, we are able to find all skew morphisms for all monolithic groups $B$ with monolith $A$ such that $B \cong \Aut(A)$, $|B : A| = 2$, and $|B| < 200000$.  This includes the cases where $B = \Sym(8)$, $\Aut(\M_{12})$ or $\PGL(2,53)$, for example. By comparison, the best method we know for determining all skew morphisms of an arbitrary finite group $A$ is computationally feasible only for $|A|$ up to $47$ (by considering all transitive permutation groups of degree $|A|$ with cyclic point-stabiliser).

Finally, we note that the methods of this section, based on Proposition~\ref{prop:notcorefree}, can be generalised to other monolithic groups, but the difficulty increases quickly with respect to $|\Aut(A):A|$.

\medskip 
\medskip
\begin{center}
{\sc Acknowledgements}
\end{center}

The authors acknowledge the use of {\sc Magma} to investigate possibilities and check examples
of groups and skew morphisms relevant to this paper.  Also the second author is grateful to the N.Z.\ Marsden Fund
(via the project UOA1626) for its support. \smallskip


\end{document}